\documentclass[10pt,a4paper]{amsart}
\usepackage{amsmath,amsthm,amssymb,amsfonts}

\usepackage[pdftex]{color}
\usepackage[bookmarks=true,hyperindex,pdftex,colorlinks,citecolor=blue,
linkcolor=blue,urlcolor=blue]{hyperref}

\usepackage[english]{babel}

\usepackage{graphicx}

\usepackage[shortlabels]{enumitem}

\usepackage{mathtools}

\theoremstyle{plain}
\newtheorem{theorem}{Theorem}[section]
\newtheorem{cor}[theorem]{Corollary}
\newtheorem{prop}[theorem]{Proposition}
\newtheorem{lemma}[theorem]{Lemma}

\theoremstyle{definition}

\newtheorem{example}[theorem]{Example}

\newtheorem{remark}[theorem]{Remark}
\newtheorem{definition}[theorem]{Definition}

\newcommand{\R}{\mathbb{R}}
\newcommand{\N}{\mathbb{N}}

\newcommand{\Lin}{\mathcal{L}}
\newcommand{\K}{\mathcal{K}}
\newcommand{\F}{\mathcal{F}}
\newcommand{\eps}{\varepsilon}

\DeclareMathOperator{\supp}{supp}
\DeclareMathOperator{\dist}{dist}

\DeclareMathOperator{\spann}{span}
\DeclareMathOperator{\diam}{diam}
\DeclareMathOperator{\co}{co}

\DeclareMathOperator{\AHSP}{AHSP}
\DeclareMathOperator{\BPBp}{BPBp}
\DeclareMathOperator{\LDBPBp}{LDirA-BPBp}

\DeclareMathOperator{\NA}{NA}

\DeclareMathOperator{\Id}{Id}
\DeclareMathOperator{\SA}{SNA}
\DeclareMathOperator{\D}{Der}
\DeclareMathOperator{\LDA}{LDirA}
\DeclareMathOperator{\DA}{DirA}
\DeclareMathOperator{\A}{A}
\DeclareMathOperator{\SAk}{SNA_\K}
\DeclareMathOperator{\Dk}{Der_\K}
\DeclareMathOperator{\LDAk}{LDirA_\K}
\DeclareMathOperator{\DAk}{DirA_\K}
\DeclareMathOperator{\Ak}{A_\K}

\newcommand{\Mol}{\operatorname{Mol}}

\newcommand{\Lip}{{\mathrm{Lip}}_0}
\newcommand{\Lipk}{{\mathrm{Lip}}_{0\K}}

\renewcommand{\leq}{\leqslant}
\renewcommand{\geq}{\geqslant}

\renewcommand{\subset}{\subseteq}

\title{Emerging notions of norm attainment for Lipschitz maps between Banach spaces}

\author[G.~Choi]{Geunsu Choi}
\address[G.~Choi]{Departament of Mathematics, POSTECH, Pohang 790-784, Republic of Korea}
\email{\texttt{chlrmstn90@postech.ac.kr}}

\author[Y.~S.~Choi]{Yun Sung Choi}
\address[Y.~S.~Choi]{Department of Mathematics, POSTECH, Pohang 790-784, Republic of Korea}
\email{\texttt{mathchoi@postech.ac.kr}}

\author[Mart\'{\i}n]{Miguel Mart\'{\i}n}
\address[Mart\'{\i}n]{Departamento de An\'{a}lisis Matem\'{a}tico, Facultad de
	Ciencias, Universidad de Granada, 18071 Granada, Spain}
\email{\texttt{mmartins@ugr.es}}

\thanks{The research of G.~Choi was supported by NRF-2015R1D1A1A09059788 and by a travel grant of the Institute of Mathematics (IEMath-GRr) of the University of Granada, Spain. The research of Y.~S.~Choi was supported by Basic Science Research Program through the National Research Foundation of Korea (NRF) funded by the Ministry of Education (NRF-2015R1D1A1A09059788 and NRF-2019R1A2C1003857). The research of M.~Mart\'{\i}n was supported by projects MTM2015-65020-P (MINECO/FEDER, UE),  PGC2018-093794-B-I00 (MCIU/AEI/FEDER, UE), and FQM-185 (Junta de Andaluc\'{\i}a/FEDER, UE)}


\date{July 3rd, 2019}


\begin{document}
	
\begin{abstract}
We classify several notions of norm attaining Lipschitz maps which were introduced previously, and present the relations among them in order to verify proper inclusions. We also analyze some results for the sets of Lipschitz maps satisfying each of these properties to be dense or not in $\Lip(X,Y)$. For instance, we characterize a Banach space $Y$ with the Radon-Nikod\'ym property in terms of the denseness of norm attaining Lipschitz maps with values in $Y$. Further, we introduce a property called the local directional Bishop-Phelps-Bollob\'as property for Lipschitz compact maps, which extends the one studied previously for scalar-valued functions, and provide some new positive results.
\end{abstract}

\maketitle
%%%%%%%%%%%%%%%%%%%%%%%%%%%%%%%%%%%%%%%%%%%%%%%%%%%%%%%%%%%%%%%%
\section{Introduction \& Preliminaries}
%%%%%%%%%%%%%%%%%%%%%%%%%%%%%%%%%%%%%%%%%%%%%%%%%%%%%%%%%%%%%%%%

It has been studied for a long time the question of whether the set $\NA(X,Y)$ of norm attaining bounded linear operators between two Banach spaces $X$ and $Y$ is dense or not in the space $\Lin(X,Y)$ of all bounded linear operators from $X$ into $Y$. As further studies were proceeded, some mathematicians got interested in asking this kind of question for the case of Lipschitz maps as well. To discuss the possibilities of norm attaining Lipschitz maps to be dense in the space of Lipschitz maps, we shall give some preliminary background information about them.

Assume that $X$ and $Y$ are \textbf{real} Banach spaces and write $\widetilde{X} = \{ (x,y) \in X^2\colon  x \neq y \}$. We denote by $\Lip(X,Y)$ the Banach space of all Lipschitz maps $f\colon   X \longrightarrow Y$ with $f(0)=0$, endowed with the norm
$$
\| f \| = \sup\left\{\frac{\| f(x) - f(y) \|}{\| x-y \|}\colon  (x,y) \in \widetilde{X}\right\}.
$$

Looking at this norm of $\Lip(X,Y)$, the most natural way for a Lipschitz map to attain its norm would be the following one \cite{Gode1,Gode2,KadMarSol}.

\begin{definition}
We say that $f \in \Lip(X,Y)$ \emph{strongly attains its norm at} $(x,y) \in \widetilde{X}$ if
$$
\frac{\| f(x)-f(y)\| }{\| x-y \|} = \| f \|.
$$
We denote by $\SA(X,Y)$ the set of all $f \in \Lip(X,Y)$ strongly attaining its norm.
\end{definition}

However, concerning the problem of the denseness of norm attaining Lipschitz maps, it is impossible to proceed further with this definition. In fact, $\SA(X,\R)$ fails to be dense in $\Lip(X,\R)$ for every Banach space $X$ (see \cite[Theorem 2.3]{KadMarSol}) and, therefore, $\SA(X,Y)$ cannot be dense in $\Lip(X,Y)$ for any Banach space $Y$ by \cite[Proposition 4.2]{ChiclanaMartin}. We refer the interested reader to the recent papers \cite{CasChiGLMarRueda,ChiclanaMartin} for the study of the denseness of strongly norm attaining Lipschitz maps defined in general metric spaces.

Recently, a few papers dealing with alternative types of norm attainment for Lipschitz maps defined on Banach spaces have appeared. Kadets, Mart\'in and Soloviova \cite[Definition 4.2]{KadMarSol} introduced another possible definition called (locally) directionally norm attaining Lipschitz function. On the other hand, Godefroy \cite{Gode2} defined other two ways in which a Lipschitz map can attain its norm. We also refer to section 8.8 of the very recent book \cite{CobzasMiculuscuNicolae-LNM} for an exposition of the results of the two aforementioned papers \cite{Gode2,KadMarSol}. Our first aim in this paper is to introduce some variations of these definitions of norm attainment and study the possible denseness of the set of Lipschitz maps attaining each of such norms. We first provide with the definitions used throughout the paper.  Definitions \ref{def:godefroy1} and \ref{def:godefroy2} were first introduced in \cite{Gode2}, and Definitions \ref{def:kms2} and \ref{def:kms1}  were first considered in \cite{KadMarSol} only for Lipschitz (real-valued) functions, which are easily extensible to the general (vector-valued) Lipschitz maps. We will use the usual notation of $B_X$, $S_X$, $X^*$ for the closed unit ball, unit sphere, and topological dual, respectively, of a Banach space $X$.

\begin{definition}[\textrm{\cite{Gode2}}]\label{def:godefroy1}
We say that $f \in \Lip(X,Y)$ \emph{attains its norm at} $x \in X$ \emph{through a derivative in the direction} $e \in S_X$ if
$$
f'(x,e) = \lim_{t \rightarrow 0} \frac{f(x+te)-f(x)}{t}\in Y \  \text{exists and satisfies that }  \| f'(x,e) \| = \| f \|.
$$
We denote by $\D(X,Y)$ the set of every $f \in \Lip(X,Y)$ which attains its norm at $x$ through a derivative in the direction $e$ for some point $x \in X$ and $e \in S_X$.
\end{definition}

Let us comment that the argument of maximal norm for a directional derivative is used in the fundamental article of Preiss \cite{Preiss-JFA} to get the existence of Fr\'{e}chet-smooth points for a Lipschitz function defined on a space with separable dual. More concretely, Preiss provides stronger versions of the result in \cite{Fitzpatrick} which showed that a Lipschitz function $f$ defined on a Banach space $X$ is Fr\'{e}chet differentiable at the point $x\in X$, if there is $e\in S_X$ such that $|f'(x,e)|=\|f\|$ (actually much less) and if the norm of $X$ is Fr\'{e}chet differentiable at $e$.

\begin{definition}[\textrm{\cite{Gode2}}]\label{def:godefroy2}
We say that $f \in \Lip(X,Y)$ \emph{attains its norm toward} $z \in Y$ if there exists $\{ (x_n,y_n) \}_{n=1}^{\infty} \subset \widetilde{X}$ such that
$$
\frac{f(x_n)-f(y_n)}{\| x_n-y_n \|} \longrightarrow z \quad \text{with } \| z \| = \| f \|.
$$
We denote by $\A(X,Y)$ the set of every $f \in \Lip(X,Y)$ which attains its norm toward $z$ for some $z \in Y$.
\end{definition}

We will see in Proposition~\ref{prop:inclusion} that this is the weakest condition among all those that we are defining here. Note also that it is proved in \cite{Gode2} (but not explicitly stated, see Example~\ref{example:Godefroy-negative_A}) that there are pairs of Banach spaces $(X,Y)$ such that $\A(X,Y)$ is not dense in $\Lip(X,Y)$. For some related results with coarse Lipschitz maps, we refer to \cite{DalLan}.

The next definitions, which extends those given in \cite{KadMarSol} for Lipschitz functions, lay in between the previously introduced notions of norm attainment.

\begin{definition}[\textrm{\cite{KadMarSol}}]\label{def:kms2}
We say that $f \in \Lip(X,Y)$ \emph{attains its norm directionally in the direction} $u \in S_X$ \emph{toward} $z \in Y$ if there exists $\{ (x_n,y_n) \}_{n=1}^{\infty} \subset \widetilde{X}$ such that
$$
\frac{f(x_n)-f(y_n)}{\| x_n-y_n \|} \longrightarrow z \quad \text{with } \| z \| = \| f \|, \quad \frac{x_n-y_n}{\| x_n-y_n \|} \longrightarrow u.
$$
We denote by $\DA(X,Y)$ the set of every $f \in \Lip(X,Y)$ which attains its norm directionally in the direction $u$ toward $z$ for some $u \in S_X$ and $z \in Y$. \end{definition}

\begin{definition}[\textrm{\cite{KadMarSol}}]\label{def:kms1}
We say that $f \in \Lip(X,Y)$ \emph{attains its norm locally directionally at the point} $\bar{x} \in X$ \emph{in the direction} $u \in S_X$ \emph{toward} $z \in Y$ if there exists $\{ (x_n,y_n) \}_{n=1}^{\infty} \subset \widetilde{X}$ such that
$$
\frac{f(x_n)-f(y_n)}{\| x_n-y_n \|} \longrightarrow z \quad \text{with } \| z \| = \| f \|, \quad \frac{x_n-y_n}{\| x_n-y_n \|} \longrightarrow u \quad \text{and} \quad x_n,y_n \longrightarrow \bar{x}.
$$
We denote by $\LDA(X,Y)$ the set of every $f \in \Lip(X,Y)$ which attains its norm locally directionally at the point $\bar{x}$ in the direction $u$ toward $z$ for some $\bar{x} \in X$, $u \in S_X$, and $z \in Y$.
\end{definition}

As a consequence of \cite[Theorem 5.3]{KadMarSol}, one obtains that $\LDA(X,\R)$ is dense in $\Lip(X,\R)$ whenever $X$ is a uniformly convex Banach space. Recall that a Banach space $X$ is said to be \emph{uniformly convex} if for every $\eps>0$, there is $\delta>0$ such that for any $x,y\in B_X$ the condition $\|x-y\|\geq \eps$ implies that $\left\|\frac{x+y}{2}\right\|\leq 1-\delta$. The best possible value of $\delta$ is denoted by $\delta_X(\eps)$ and called the modulus of convexity of $X$. As far as we know, the cited consequence of \cite[Theorem 5.3]{KadMarSol} is the only known positive result on the denseness of different kind of norm attaintment for Lipschitz maps defined on a Banach space.

To get shaper results, we also deal in this paper with Lipschitz compact maps. We say $f \in \Lip(X,Y)$ is a \emph{Lipschitz compact map} if the set
$$
\operatorname{Slope}(f):=\left\{ \frac{f(x)-f(y)}{\| x-y \|} \colon  (x,y) \in \widetilde{X} \right\} \subset Y
$$
(which is called the \emph{set of slopes} or the \emph{Lipschitz image} of $f$) is relatively compact in $Y$, and denote by $\Lipk(X,Y)$ the space of all Lipschitz compact maps from $X$ into $Y$. Observe that if $Y$ is finite-dimensional then every Lipschitz map is indeed a Lipschitz compact map, whereas we cannot say that when $X$ is finite-dimensional. We refer to \cite[\S 8.6]{CobzasMiculuscuNicolae-LNM} and \cite{JimenezSepulcreVillegas} for background.
Now we apply the five definitions of norm attainment to the set of Lipschitz compact maps to get the corresponding norm attaining sets: given Banach spaces $X$, $Y$, we write
\begin{align*}
\SAk(X,Y)&:=\SA(X,Y)\cap \Lipk(X,Y), & \Dk(X,Y)&:=\D(X,Y)\cap \Lipk(X,Y),\\
\LDAk(X,Y)&:=\LDA(X,Y)\cap \Lipk(X,Y), & \DAk(X,Y)&:=\DA(X,Y)\cap \Lipk(X,Y),\\
\Ak(X,Y)&:=\A(X,Y)\cap \Lipk(X,Y).
\end{align*}

Let us comment now what happens with all the introduced definitions when the Lipschitz map is actually a linear operator. Given Banach spaces $X$ and $Y$, we denote by $\Lin(X,Y)$ the space of all bounded linear operator from $X$ into $Y$, endowed with the operator norm. It is clear that $\Lin(X,Y)\subseteq \Lip(X,Y)$ with equality of the norms. Recall that $T\in \Lin(X,Y)$ \emph{attain its norm} (as a linear operator) at $x_0 \in S_X$ if
$$
\| T \| = \sup_{x \in B_X} \| Tx \| = \| Tx_0 \|,
$$
and $\NA(X,Y)$ denotes the set of all $T \in \Lin(X,Y)$ which attains its norm. We summarize in the next result the relations between the different notions of norm attainment that we have introduced when they are applied to bounded linear operators. We denote by $\mathcal{K}(X,Y)$ the space of all compact linear operators from $X$ into $Y$.

\begin{remark}\label{remark-NA-linear}
{\slshape Let $X$, $Y$ be Banach spaces.}
\begin{itemize}
\item[(a)]
$\SA(X,Y)\cap \Lin(X,Y)  = \D(X,Y)\cap \Lin(X,Y)    = \LDA(X,Y)\cap \Lin(X,Y) = \DA(X,Y)\cap \Lin(X,Y)$
{\slshape and these sets coincide with $\NA(X,Y)$.}
\item[(b)] $\mathcal{K}(X,Y)\subset \A(X,Y)$.
\end{itemize}
In fact, (a) follows immediately from the definitions, contininuity and linearity of the elements of $\Lin(X,Y)$. To get (b), fix $T\in \mathcal{K}(X,Y)$ and take $z\in \overline{T(S_X)}$ with $\|z\|=\|T\|$, which is possible due to the compactness of $T$. Now, we may consider a sequence $\{x_n\}$ in $S_X$ such that $T(x_n)\longrightarrow  z$ and then the linearity of $T$ gives that $T\in \A(X,Y)$.
\end{remark}

So far we have introduced five definitions of norm attainment for Lipschitz maps. Our aim in Section~\ref{chapter:inclusion} is to show the inclusion relations between the sets of norm attainment. We first show that for arbitrary Banach spaces $X$ and $Y$, they partially form a chain of subsets: $$\D(X,Y) \subset \LDA(X,Y) \subset \DA(X,Y) \subset \A(X,Y) \subset \Lip(X,Y)$$ and that $$\SA(X,Y) \subset \DA(X,Y).$$  When $Y$ has the Radon-Nikod\'{y}m property, we show that
$$
\SA(X,Y)\subset \D(X,Y)
$$
and that this inclusion is not true in general. We show examples that all inclusions can be proper, and characterize when the equalities hold, getting some characterizations of finite dimensionality. For Lipschitz compact maps, the situation is easier, as we will see that
$$
\SAk(X,Y)\subset \Dk(X,Y) \subset \LDAk(X,Y) \subset \DAk(X,Y) \subset \Ak(X,Y)=\Lipk(X,Y),
$$
and also that each inclusion can be proper. We analyze the cases where the equalities occur, getting some more characterizations of finite dimensionality.

In Section~\ref{chapter:positiveldak}, we deal with the problem of determining when the different sets of norm attaining Lipschitz maps are dense. We show that $\D(\R,Y)$ is dense in $\Lip(\R,Y)$ if and only if  $Y$ has the Radon-Nikod\'{y}m property. Moreover, if $\D(X,Y)$ is dense in $\Lip(X,Y)$ for some $X$, then $Y$ must have the Radon-Nikod\'{y}m property. On the other hand, it is also shown that $\Dk(\R,Y)$ is dense in $\Lipk(\R,Y)$ for all Banach spaces $Y$. Besides, we provide some sufficient conditions to get that $\A(X,Y)$ is dense in $\Lip(X,Y)$.

In order to discuss the content of Section~\ref{chapter:dlbpbp}, we need some notions given in \cite{AcoAroGarMae}. Acosta, Aron, Garc\'ia and Maestre introduced the \emph{Bishop-Phelps-Bollob\'as property} ($\BPBp$ for short) \emph{for} (\emph{linear}) \emph{operators}, a name given to those pairs of Banach spaces $(X,Y)$ satisfying the following: for every $\eps>0$, there exists $\eta>0$ such that whenever $T \in S_{\Lin(X,Y)}$ and $x \in S_X$ satisfy $\|Tx\|>1-\eta$, there exist $S \in S_{\Lin(X,Y)}$ and $y \in S_X$ such that $\|Sy\|=1$, $\|S-T\|<\eps$ and $\|y-x\|<\eps$. If $T$ and $S$ above are compact, we get the analogous definition of $\BPBp$ \emph{for compact operators}.

Banach spaces with some geometrical properties play an important role as a range space in the viewpoint of $\BPBp$ for operators. A Banach space $X$ is said to have \emph{property $\beta$}, which was first introduced by Lindenstrauss in \cite{Lin}, if there exist a collection $\{(z_i, z_i^*)\}_{i\in I} \subset S_X \times S_{X^*}$ and a constant $0 \leq \lambda <1$ satisfying (1) $|z_i^*(z_i)| = 1$ for all $i \in I$, (2) $|z_i^*(z_j)| \leq \lambda < 1$ if $i \neq j$, and (3) $\|z\| = \sup_{i \in I} |z_i^*(z)|$ for any $z \in X$. For instance, finite-dimensional spaces with polyhedral unit balls, $c_0$ and $\ell_\infty$ have property $\beta$. When $Y$ has property $\beta$, $(X,Y)$ has the $\BPBp$ for operators and the $\BPBp$ for compact operators for arbitrary domain space $X$ (see \cite[Theorem~2.2]{AcoAroGarMae} and \cite[Example 1.5]{DanGarMaeMar}).

Our aim in Section~\ref{chapter:dlbpbp} is to extend results in \cite{KadMarSol} about some version of the Bishop-Phelps-Bollob\'{a}s property for scalar-valued Lipschitz functions to vector-valued cases. Let us present the main definition which extends \cite[Definition 4.3]{KadMarSol} to vector-valued maps. Note that $[x,y]$ denotes the segment joining $x$ and $y$.

\begin{definition}\label{def:sldBPBp}
A pair of Banach spaces $(X,Y)$ is said to have the \emph{local directional Bishop-Phelps-Bollob\'as property} (in short, $\LDBPBp$) \emph{for Lipschitz maps} if for every $\eps>0$, there exists $\eta >0$ such that whenever $f \in S_{\Lip(X,Y)}$  and $(x,y) \in \widetilde{X}$ satisfy
$$
\frac{\| f(x)-f(y) \|}{\| x-y \| } > 1 - \eta,
$$
there exist $g \in S_{\Lip(X,Y)}$, $z \in S_Y$, $u \in S_X$ and $\bar{x} \in X$ such that $g$ attains its norm locally directionally at the point $\bar{x}$ in the direction $u$ toward $z$, $\| g -f \| < \eps$, $\bigl\| u - \frac{x-y}{\| x - y \|} \bigr\| < \eps$ and $\dist (\bar{x}, [x,y] ) < \eps$.

If $f$ and $g$ above are Lipschitz compact, we get the analogous definition of the $\LDBPBp$ \emph{for Lipschitz compact maps}.
\end{definition}

Observe that if a pair of Banach spaces $(X,Y)$ has the $\LDBPBp$ for Lipschitz maps, then $\LDA(X,Y)$ is dense in $\Lip(X,Y)$. Analogously, if $(X,Y)$ has the $\LDBPBp$ for Lipschitz compact maps, then $\LDAk(X,Y)$ is dense in $\Lipk(X,Y)$.

If $X$ is a uniformly convex Banach space, it is shown in \cite[Theorem 5.3]{KadMarSol} that the pair $(X,\R)$ has the $\LDBPBp$ for Lipschitz maps. We will show in Section~\ref{chapter:dlbpbp} that if $X$ is a uniformly convex Banach space, $Y$ is a Banach space, and the pair $(\F(X),Y)$ has the $\BPBp$ for compact operators, then the pair $(X,Y)$ has the $\LDBPBp$ for Lipschitz compact maps. In particular, this applies for all uniformly convex spaces $X$, if $Y$ has property $\beta$, or if $Y^*$ is isometrically isomorphic to some $L_1(\mu)$-space like $Y=C_0(L)$, where $L$ is a locally compact Hausdorff space. In the case where $X$ is a Hilbert space $H$, we also get a slightly different property for the pair $(H,Y)$ under the same assumptions on the space $Y$.

The techniques which will be used to get results for the $\LDBPBp$ for Lipschitz compact maps require the notion of the so-called Lipschitz-free spaces. The rest of this introduction is devoted to present the necessary background. For a Banach space $X$, we can associate to each $x \in X$ an element $\delta_{x} \in \Lip(X,\R)^{*}$, which is just the evaluation map $\delta_{x}(f) = f(x)$ for every $f\in \Lip(X,\R)$. The \emph{Lipschitz-free space} over $X$ is defined as
$$
\F(X):= \overline{\spann}^{\| \cdot \|}{\{ \delta_{x}\colon  x \in X\}} \subseteq \Lip(X, \R)^{ *}.
$$
Note that the map $x \longmapsto \delta_{x}$ establishes an isometric (non-linear) embedding $X \hookrightarrow \F(X)$, because $\|\delta_{x}-\delta_{y}\|=\|x-y\|$ for all $x,y\in X$. We refer the reader to the paper \cite{Gode1}  and the books \cite{CobzasMiculuscuNicolae-LNM} and \cite{Weaver} for more information and background. The main features of the Lipschitz-free space that we are going to use here are contained in the following result, which is nowadays considered folklore.

\begin{lemma}\label{Lemm:elementarypropertiesF(X)}
Let $X$, $Y$ be Banach spaces.
\begin{enumerate}
\item[\textup{(a)}] For every $f\in \Lip(X,Y)$ there exists a unique linear operator $T_f \in \Lin(\F(X),Y)$ such that $T_{f}\circ \delta = f$ with $\|T_{f}\| = \|f\|$. Moreover, this correspondence defines an isometric isomorphism between the space $\Lip(X,Y)$ and $\Lin(\F(X),Y)$. In particular, $\F(X)^* = \Lip(X,\R)$.
\item[\textup{(b)}] $f \in \Lipk(X,Y)$ if and only if $T_f \in \K(\F(X),Y)$.
\item[\textup{(c)}] The set
\begin{equation*}
 \Mol(X) := \left\{ \frac{\delta_{x} - \delta_{y}}{\| x - y\|}\colon  (x,y) \in \widetilde{X} \right\}\subseteq \F(X)
\end{equation*}
is rounded and norming for $\F(X)^*$, i.e.\ $B_{\F(X)} = \overline{\co}{(\Mol(X))}$, where $\overline{\co}(\Mol(X))$ denotes the closed convex hull of $\Mol(X)$.
\item[\textup{(d)}] $\F(\R)$ is isometrically isomorphic to $L_1(\R)$ through the map $\delta_t \longmapsto \chi_{[0,t]}$ or, equivalently, $\Lip(\R,\R) = L_\infty(\R)$ through the differentiation map.
\end{enumerate}
\end{lemma}

%%%%%%%%%%%%%%%%%%%%%%%%%%%%%%%%%%%%%%%%%%%%%%%%%%%%%%%%%%%%%%%%
\section{Relations among the different notions of norm attaintment} \label{chapter:inclusion}
%%%%%%%%%%%%%%%%%%%%%%%%%%%%%%%%%%%%%%%%%%%%%%%%%%%%%%%%%%%%%%%%

We begin this section with presenting the inclusion relations among the different kinds of sets of norm attaining Lipschitz maps which we have presented in the introduction.

\begin{prop}\label{prop:inclusion}
Let $X$ and $Y$ be Banach spaces.
\begin{itemize}
\item[\textup{(a)}] $\D(X,Y) \subset \LDA(X,Y) \subset \DA(X,Y) \subset \A(X,Y) \subset \Lip(X,Y)$.
\item[\textup{(b)}] $\SA(X,Y) \subset \DA(X,Y)$.
\item[\textup{(c)}] If $\dim(X)<\infty$, then $\DA(X,Y)=\A(X,Y)$.
\item[\textup{(d)}] If $\dim(Y)<\infty$, then  $\A(X,Y) = \Lip(X,Y)$.
\item[\textup{(e)}] If $Y$ has the Radon-Nikod\'{y}m property, then $\SA(X,Y)\subset \D(X,Y)$.
\end{itemize}
\end{prop}

We need the following easy consequence of \cite[Lemma 2.2]{KadMarSol}.

\begin{lemma}\label{lemma:SNAinLDA}
Let $X$, $Y$ be Banach spaces. If $f\in \Lip(X,Y)$ strongly attains its norm at $(x,y)\in \widetilde{X}$, then
  $
  \|f(v)-f(w)\|=\|f\|\,\|v-w\|
  $
  for every $(v,w)\in \widetilde{[x,y]}$.
\end{lemma}

\begin{proof}
Take $y^*\in S_{Y^*}$ such that
$$
y^*\left(\frac{f(x)-f(y)}{x-y}\right)=\left\|\frac{f(x)-f(y)}{x-y}\right\|=\|f\|.
$$
This shows that the Lipschitz function $\psi=y^*\circ f\in\Lip(X,\R)$ strongly attains its norm at the pair $(x,y)$, so by \cite[Lemma 2.2]{KadMarSol}, we have that
$$
 \frac{|\psi(v) - \psi(w)|}{\|v-w\|} =\|\psi\|=\|f\|
$$
for every $(v,w)\in \widetilde{[x,y]}$. This gives the result immediately.
\end{proof}

We also need the following well-known result for which we will include some comments on how it can be proved. It will be also useful later on.

\begin{lemma}\label{lemma:RNP-compact-derivable}
Let $Y$ be a Banach space.
\begin{enumerate}
\item[\textup{(a)}] If $g \in \Lip(\R,Y)$ and either $Y$ has the Radon-Nikod\'{y}m property or $g$ is Lipschitz compact, then there is $\varphi \in L_\infty(\R,Y)$ such that
$$
g(t)=\int_{0}^t \varphi(s)\,ds \qquad \text{\emph{for} } t\in \R.
$$
Note that, in this case, $g$ is differentiable almost everywhere and $\varphi$ coincides almost everywhere with $g'$. Moreover, $\|g\|=\|\varphi\|_\infty$.
\item[\textup{(b)}] Conversely, if $\psi\in L_\infty(\R,Y)$ and we define $h\colon \R\longrightarrow Y$ by
$$
h(t)=\int_0^t \psi(s)\, ds \qquad \text{\emph{for} } t\in \R,
$$
then $h\in \Lip(\R,Y)$, $h$ is differentiable a.e., $h'=\psi$ a.e., and $\|h\|=\|\psi\|_\infty$. Moreover,  $\psi(\R)\subset K$ a.e.\  for some compact subset $K$ of $Y$ if and only if $h\in \Lipk(\R,Y)$.
\end{enumerate}
\end{lemma}

If $Y$ has the Radon-Nikod\'{y}m property, then the first assertion in (a) follows easily from the proof of \cite[Theorem 5.21]{BenyaLinden}, where it is stated for Lipschitz maps defined on bounded intervals, but the result can be extended to those defined in the whole $\R$. For a Lipschitz compact map, we note that the ideas in the proof of \cite[Theorem 5.21]{BenyaLinden} are valid for a set with the Radon-Nikod\'{y}m property, and the rest of the proof is the same. For the sake of completeness, we would like to provide a direct proof using Lipschitz-free spaces.

\begin{proof}
(a) Suppose $Y$ has the Radon-Nikod\'ym property. Then, $T_g \in \Lin(L_1(\R),Y)$ defined as in Lemma~\ref{Lemm:elementarypropertiesF(X)}.(a) is representable \cite[Theorem III.1.5]{DieUhl}. That is, there exists $\varphi \in L_\infty(\R,Y)$ such that
$$
T_g(f) = \int f\varphi(s)\, ds \qquad \text{for every } f \in L_1(\R).
$$
Hence, if we put $f = \chi_{[0,t]}$, we get
$$
g(t) = T_g(\chi_{[0,t]}) = \int_{0}^t \varphi(s)\, ds \qquad \text{for } t \in \R
$$
by the isometric correspondence given in Lemma \ref{Lemm:elementarypropertiesF(X)}.

If we assume rather that $g \in \Lipk(\R,Y)$, then we can also deduce that $T_g \in \K(L_1(\R),Y)$ is representable due to Lemma \ref{Lemm:elementarypropertiesF(X)}.(b) and \cite[Theorem III.2.2]{DieUhl}. The rest of the proof is identical to the previous case.

(b) Only the `moreover' part requires a comment: if $\psi(\R)\subset K$ a.e.\ for some compact subset $K$ of $Y$, then
$\operatorname{Slope}(h)\subset \overline{\co}(K)$ (see \cite[Proposition~1.6.9.iv]{CobzasMiculuscuNicolae-LNM}, for instance), so $h\in\Lipk(\R,Y)$ as desired. Conversely, if $h\in \Lipk(\R,Y)$, then the conclusion easily follows from that $\psi(t)\in \overline{\operatorname{Slope}(h)}$ a.e.\ for $t\in \R$.
\end{proof}

We now provide the pending proof of Proposition \ref{prop:inclusion}.

\begin{proof}[Proof of Proposition \ref{prop:inclusion}]
(a) To prove that $\D(X,Y) \subset \LDA(X,Y)$, let $f \in \D(X,Y)$. Then, there exist $x \in X$ and $e \in S_X$ such that $f'(x,e)$ exists and $\| f'(x,e) \| = \| f \|.$ Put $(x_n,y_n)=(x+\frac{e}{n},x)$ for each $n \in \N$. As $n$ tends to $\infty$, we have that
$$
\frac{f(x_n)-f(y_n)}{\| x_n-y_n \|} \longrightarrow f'(x,e), \quad \frac{x_n-y_n}{\| x_n-y_n \|} \longrightarrow e \quad \text{and} \quad x_n,y_n \longrightarrow x.
$$
The rest of the inclusions are obvious from their definitions.

(b) If $f\in \Lip(X,Y)$ strongly attains its norm at $(x,y)\in \widetilde{X}$, it is immediate that $f$ attains its norm directionally in the direction $u=\frac{x-y}{\|x-y\|}$ using the constant sequence.

Assertion (c) is immediate from the compactness of $S_X$ and the same happens for (d) from the compactness of all closed bounded subsets of $Y$.

(e) Suppose that $f\in \Lip(X,Y)$ strongly attains its norm at $(x,y)$, write $u=\frac{x-y}{\|x-y\|}\in S_X$ and consider $g\colon  \R \longrightarrow  Y$ defined by $g(t)=f(y+tu) - f (y)$ for $t \in \R$. It is clear that $g \in \SA(\R,Y)$ with $\|g\|= \|f\|$. As $Y$ has the Radon-Nikod\'{y}m property, Lemma \ref{lemma:RNP-compact-derivable}.(a) gives that $g$ is differentiable almost everywhere and, moreover,
$$
g(t)=\int_0^t g'(s)\,ds \qquad \text{for } t \in \R.
$$
Since $g$ strongly attains its norm at $(0,\|x-y\|)$, Lemma \ref{lemma:SNAinLDA} shows that
$$
\|g(t)-g(s)\|=\|g\|\,|t-s| \qquad \text{for } 0 \leq t,s \leq \|x-y\|,
$$
which implies that
\begin{align*}
  \|g\| & = \frac{1}{|t|}\bigl\|g(t)-g(0)\bigr\| \leq \frac{1}{|t|} \int_0^t \|g'(s)\|\,ds \leq \|g\| \qquad \text{for } 0 < t \leq \|x-y\|.
  \end{align*}
Thus, $\|g'(s)\|=\|g\|$ almost everywhere $0 < s \leq \|x-y\|$. It is enough to consider any $0 < s_0 \leq \|x-y\|$ for which $\|g'(s_0)\|=\|g\|=\|f\|$, write $x_0=y+s_0u\in X$, and observe that $f'(x_0,u)=g'(s_0)$ to conclude that $f\in \D(X,Y)$.
\end{proof}

The next result shows that the equality in Proposition \ref{prop:inclusion}.(c) only holds when the domain space is finite-dimensional.

\begin{prop}\label{prop:Xinfdim-DAneqA}
Let $X$ be an infinite-dimensional Banach space. Then, $\DA(X,Y)\neq \A(X,Y)$ for every nontrivial Banach space $Y$.
\end{prop}

\begin{proof}
If $X$ is infinite dimensional, it is shown in \cite[Lemma 2.2]{M-M-P} that there exists $T\in \Lin(X,c_0)$ which does not attains its norm as a linear operator. If we fix any $y_0\in S_{Y}$ and define $f\colon X \longrightarrow  Y$ by $f(x)=\|T(x)\|y_0$, then it is evident that $f\in \A(X,Y)$ with $\|f\|=\|T\|$. On the other hand, $f\notin \DA(X,Y)$ by the almost same argument as in \cite[Lemma~3.2]{KadMarSol}. Indeed, assume that $f \in \DA(X,Y)$. It follows that there exist $\{(x_n,y_n)\}_{n=1}^\infty\subset \widetilde{X}$ and $u \in S_X$ such that
$$
\frac{\|f(x_n)-f(y_n)\|}{\|x_n-y_n\|} \longrightarrow \|f\| \quad \text{and} \quad \frac{x_n-y_n}{\|x_n-y_n\|} \longrightarrow u.
$$
But this is impossible, because
\begin{align*}
\frac{\|f(x_n)-f(y_n)\|}{\|x_n-y_n\|} &=\frac{\bigl\|\|T(x_n)\| y_0 - \|T(y_n)\| y_0\bigr\|}{\|x_n-y_n\|} \leq \frac{\|T(x_n)-T(y_n)\|}{\|x_n-y_n\|} \longrightarrow \| T(u) \|
\end{align*}
and $T$ does not attain its norm at $u \in S_X$ as a linear operator.
\end{proof}

The next example shows that the inclusion given in Proposition \ref{prop:inclusion}.(e) can be false for a range space $Y$ without the Radon-Nikod\'{y}m property.

\begin{example}
{\slshape Consider a Lipschitz map $f\colon \R \longrightarrow  L_1(\R)$ given by
\begin{displaymath}
f(t)=\left\{\begin{array}{@{}cl}
\displaystyle \phantom{.} 0 & \text{\emph{if} } t \leq 0 \\\\
\displaystyle \phantom{.} \chi_{[0,t]} & \text{\emph{if} } t>0.
\end{array} \right.
\end{displaymath}
Then, $f\in \SA(\R,L_1(\R))$ but $f\notin \LDA(\R,L_1(\R))$.}

Indeed, for all $0<s<t$ we have that
$$
\left\|\frac{f(t)-f(s)}{t-s}\right\|=1,
$$
so $f\in \SA(\R,L_1(\R))$. On the other hand, suppose that $f$ attains its norm locally directionally at the point $\bar{t} \in \R$ for some sequence $\{(t_n,s_n)\}_{n=1}^\infty \subset \widetilde{\R}$. Since $|t_n-s_n| \longrightarrow 0$, it is immediate that the sequence
$$
\left\{\frac{f(t_n)-f(s_n)}{t_n-s_n}\right\}_{n=1}^\infty \subset L_1(\R)
$$
either converges to $0$ or does not converge in $L_1(\R)$. It follows that $f\notin \LDA(\R,L_1(\R))$.
\end{example}

Example \ref{example:r} and Proposition \ref{prop:aneq} below, together with Proposition \ref{prop:Xinfdim-DAneqA} show that all the inclusions given in assertions (a), (b), and (e) of Proposition~\ref{prop:inclusion} can be proper.

\begin{example}\label{example:r}
{\slshape We have that}
$$
\SA(\R,\R) \subsetneqq \D(\R,\R) \subsetneqq \LDA(\R,\R) \subsetneqq \DA(\R,\R).
$$
\end{example}

\begin{proof}
Note that $\SA(\R,\R)$ is the set of all functions $f \in \Lip(\R,\R)$ which contain a line segment with slope either $\| f \|$ or $-\| f \|$ in its graph. To see that $\SA(\R,\R) \subsetneqq  \D(\R,\R)$, consider $f(t) = \sin t$, whose graph contains no line segment but $f'(0)=\| f \| = 1$.

To see that $\D(\R,\R) \subsetneqq  \LDA(\R,\R)$, define $\displaystyle  g(t)=\int_0^t \varphi(s)\,ds$, where $\varphi \in L_\infty(\R)$ is given by
\begin{displaymath}
\varphi(t)=\left\{\begin{array}{@{}cl}
\displaystyle \phantom{.} 1-2^{-n} & \text{if } 2^{-2^n} < t < 2^{-2^n}+2^{-2^{n+1}} \text{ for each } n \in \N, \\\\
\displaystyle \phantom{.} 0 & \text{otherwise}.
\end{array} \right.
\end{displaymath}
Clearly, $g \in \Lip(\R,\R)$ and $\|g\|=1$. If we put $(t_n,s_n)=(2^{-2^n},2^{-2^n}+2^{-2^{n+1}})$ for each $n \in \N$, then we can easily see that $g \in \LDA(\R,\R)$ from
$$
\frac{g(t_n)-g(s_n)}{t_n-s_n} \longrightarrow 1 \quad \text{and} \quad t_n,s_n \longrightarrow 0.
$$
On the other hand, we have
$$
\lim_{t \to t_0} \left| \frac{g(t)-g(t_0)}{t-t_0} \right| <1 \quad \text{if } t_0 \neq 0 \quad \text{and} \quad \lim_{t\to0} \left|\frac{g(t)-g(0)}{t-0} \right| = 0.
$$
Indeed, given $\eps>0$ choose $n \in \N$ so that $2^{1-2^n} < \eps$. Fix any point $0<t<2^{-2^{n-1}}$. We can find $n_0 \geq n$ such that $2^{-2^{n_0}} \leq t < 2^{-2^{n_0-1}}$. Then,
$$
\frac{g(t)}{t} = \frac{1}{t}\int_0^t g'(s)\, ds \leq \frac{1}{t}\sum_{k=n_0}^{\infty} 2^{-2^{k+1}} \leq \sum_{k=n_0}^{\infty} 2^{-2^{k+1}} \cdot 2^{2^{n_0}} \leq 2^{1-2^{n_0}} < \eps,
$$
which shows that $\lim_{t\to0^+} \frac{g(t)}{t} = 0$. The rest is clear.

Finally, to see that the inclusion $\LDA(\R,\R) \subset \DA(\R,\R)$ is proper, consider the function $h(t)= \sqrt{1+t^2}-1$. A simple calculation shows that $\lim_{t \to \infty} h'(t)= \| h \| = 1$ while $h'$ is continuous and $|h'(t)|<1$ for any $t \in \R$, so $h\notin \LDA(\R,\R)$, as desired.
\end{proof}

The next proposition characterizes the finite dimensionality of a Banach space $Y$ in term of the set $A(X,Y)$, and shows that the inclusion $\A(X,Y) \subset \Lip(X,Y)$ is proper in many cases.

\begin{prop}\label{prop:aneq}
Let $Y$ be an infinite dimensional Banach space. Then $\A(X,Y) \subsetneqq \Lip(X,Y)$ for every nontrivial Banach space $X$.
\end{prop}

\begin{proof}
Choose a basic sequence of distinct vectors $\{v_j\colon  j \in \N\} \subset S_Y$ and consider $f\colon  X \longrightarrow  Y$ defined by
$$
f(x)= \sum_{j=1}^\infty \frac{j}{j+1} s_{j} (\|x\|) v_j,
$$
where each $s_j\colon  \R \longrightarrow  \R$ is given as
\begin{displaymath}
s_j(t)=\left\{\begin{array}{cl}
\displaystyle 0 & \text{if } t < 2j-2, \\
\displaystyle t-2j+2 & \text{if } 2j-2 \leq t < 2j-1, \\
\displaystyle -t+2j & \text{if } 2j-1 \leq t < 2j, \\
\displaystyle 0 & \text{if } 2j \leq t.
\end{array} \right.
\end{displaymath}
To see that $\|f\| \leq 1$, we claim that
$$
\|f(x)-f(y)\| \leq \bigl| \|x\|-\|y\|\bigr|,
$$
for any $x,y \in X$, and then the conclusion follows from the fact that $\bigl| \|x\|-\|y\|\bigr| \leq \|x-y\|$. Indeed, given any $x \in X$, we denote by $j_x$ the unique corresponding $j \in \N$ of $x$ such that $2j-2 \leq \|x\| < 2j$. Then, from the construction it is obvious that $s_j(x)=0$ if $j \neq j_x$. Let $x,y \in X$ be given. First, suppose that $j_x=j_y$. Then, we have
$$
\|f(x)-f(y)\| \leq \bigl| s_{j_x}(\|x\|) - s_{j_x}(\|x\|) \bigr| \leq \bigl| \|x\| - \|y\| \bigr|,
$$
because $\|s_{j_x}\| \leq 1$. So it remains to show when $\|x\|<2j_x \leq \|y\|$. But note that
\begin{align*}
\|f(x)-f(y)\| \leq s_{j_x}(\|x\|) + s_{j_y}(\|y\|) & = \bigl[ s_{j_x}(\|x\|)-s_{j_x}(2j_x) \bigr]+\bigl[s_{j_y}(\|y\|)-s_{j_y}(2j_x)\bigr] \\
& \leq \bigl( 2j_x - \|x\|\bigr) + \bigl( \|y\| - 2j_x \bigr) = \bigl| \|x\| - \|y\| \bigr|,
\end{align*}
which proves the claim.

Now, fix $x_0\in S_X$ and write $(x_n,y_n)=(2nx_0,(2n+1)x_0)\in \widetilde{X}$ for each $n \in \N$. Then,
$$
\| f \| \geq \sup_{n} \frac{\| f(x_n)-f(y_n) \|}{\|x_n-y_n\|} = 1,
$$
hence $\|f\| =1.$
But $f$ cannot attain its norm toward any $z \in S_Y$. Indeed, suppose that $f$ attains its norm toward $z \in S_Y$ for some sequence $\{(x_n, y_n)\}_{n=1}^{\infty} \subset \widetilde{X}$. Up to a subsequence, we may suppose that $\|x_n\|<\|y_n\|$ for every $n \in \N$. Observe that $z \in \overline{\spann}^{\|\cdot\|} \{v_j\colon j \in \N\}$ and thus, being a basic sequence, we get that $z=\sum_{n=1}^{\infty} a_n v_n$ for suitable sequence $\{a_n\}$ of scalars. Without loss of generality, let $a_1 \neq 0$. Then $\|x_n\| <2$ for sufficiently large $n \in \N$. Otherwise, $f(x_n)-f(y_n)$ would have zero coefficient on $v_1$. Finally, we claim that
$$
\frac{\| f(x_n)-f(y_n) \|}{\|x_n-y_n\|} \leq \frac{2}{3} \qquad \text{if } \|x_n\|<2,
$$
which will end up with a contradiction. If $\|y_n\| \leq 2$, then it is clear that $$\| f(x_n)-f(y_n) \| \leq \frac{1}{2} \bigl| \|x_n\|-\|y_n\| \bigr|.$$
So we may assume that $\|y_n\|>2$. Fix any $x_0 \in S_X$ and, by a simple calculation, we can see that
\begin{align*}
\frac{\|f(x_n)-f(y_n)\|}{\|x_n-y_n\|} &\leq \max \left\{ \frac{\|f(x_n)-f(2x_0)\|}{\bigl| \|x_n\|-2 \bigr|}, \frac{\|f(2x_0)-f(y_n)\|}{\bigl| 2-\|y_n\| \bigr|} \right\} \\
& \leq \max \left\{ \frac{1}{2}, \frac{\|f(y_n)\|}{\bigl| 2-\|y_n\| \bigr|} \right\}.
\end{align*}
Hence, it suffices to check that $\frac{\|f(y_n)\|}{|2-\|y_n\||} \leq 2/3$. If $\|y_n\| < 4$, then $j_{2x_0}=j_{y_n}$ which ensures that $\frac{\|f(y_n)\|}{| 2-\|y_n\| |} \leq 2/3$. If $\|y_n\| \geq 4$, then we have $\bigl| 2-\|y_n\| \bigr| \geq 2$, so that $\frac{\|f(y_n)\|}{|2-\|y_n\||} \leq 1/2$.
\end{proof}

As an immediate consequence of Propositions \ref{prop:Xinfdim-DAneqA} and \ref{prop:aneq}, we get the following characterization of the finite dimensionality of both $X$ and $Y$, simultaneously.
\begin{cor}
Let $X$, $Y$ be nontrivial Banach spaces. Then, $\DA(X,Y)=\Lip(X,Y)$ if and only if both $X$ and $Y$ are finite-dimensional.
\end{cor}

Let us now discuss the inclusion relations for Lipschitz compact maps.

\begin{prop}\label{prop:inclusion-compact}
Let $X$ and $Y$ be Banach spaces. Then
\begin{enumerate}
\item[\textup{(a)}] $\SAk(X,Y)\subset \Dk(X,Y) \subset \LDAk(X,Y) \subset \DAk(X,Y) \subset \Ak(X,Y)$.
\item[\textup{(b)}] $\Ak(X,Y)=\Lipk(X,Y)$.
\end{enumerate}
\end{prop}

Note that all the inclusions in Proposition \ref{prop:inclusion-compact}.(a) can be proper as shown in Proposition \ref{prop:Xinfdim-DAneqA} and Example \ref{example:r}.

\begin{proof}
(a) Since all the inclusions but the first one are direct consequences of Proposition \ref{prop:inclusion}.(a), it remains to show that $\SAk(X,Y)\subset \Dk(X,Y)$. Indeed, we just follow the proof of Proposition \ref{prop:inclusion}.(e), taking into account that $g$ is Lipschitz compact, and we can also apply Lemma \ref{lemma:RNP-compact-derivable}.(a).

(b) Let $f \in \Lipk(X,Y)$ be given. Choose a sequence $\{ (x_n,y_n) \}_{n=1}^{\infty} \subset \widetilde{X}$ such that
$$
\frac{\| f(x_n)-f(y_n) \|}{\| x_n-y_n \|} \longrightarrow \| f \|.
$$
Note that $\left\{ \frac{f(x_n)-f(y_n)}{\| x_n-y_n \|} \colon  n \in \N \right\} \subset \operatorname{Slope}(f)$ is relatively compact in $Y$. So, passing to a subsequence, we can find $z \in Y$ such that
$$
\frac{f(x_n)-f(y_n)}{\| x_n-y_n \|} \longrightarrow z
$$
and, of course, $\| z \| = \| f \|$ by our election of the sequence $\{(x_n,y_n)\}_{n=1}^\infty \subset \widetilde{X}$.
\end{proof}

In particular, we get another characterization of the finite dimensionality of $X$.

\begin{cor}
Let $X$ be a Banach space. Then, the following are equivalent:
\begin{enumerate}
\item[\emph{(i)}] $X$ is finite-dimensional.
\item[\emph{(ii)}] $\DAk(X,Y) = \Lipk(X,Y)$ for every Banach space $Y$.
\item[\emph{(iii)}] There is a nontrivial Banach space $Y$ such that $\DAk(X,Y) = \Lipk(X,Y)$.
\end{enumerate}
\end{cor}

\begin{proof}
(i)$\Rightarrow$(ii) is given by Propositions \ref{prop:inclusion}.(c) and \ref{prop:inclusion-compact}.(b). (ii)$\Rightarrow$(iii) is immediate. Finally, (iii)$\Rightarrow$(i) follows from the proof of Proposition \ref{prop:Xinfdim-DAneqA}, because the map $f$ defined there is clearly compact.
\end{proof}

%%%%%%%%%%%%%%%%%%%%%%%%%%%%%%%%%%%%%%%%%%%%%%%%%%%%%%%%%%%%%%%%
\section{Some results on denseness of norm attaining Lipschitz maps} \label{chapter:positiveldak}
%%%%%%%%%%%%%%%%%%%%%%%%%%%%%%%%%%%%%%%%%%%%%%%%%%%%%%%%%%%%%%%%
Our main results in this section deal with the denseness of Lipschitz maps defined on $\R$ attaining their norm through a derivative. We recall that $\SA(X,Y)$ is never dense in $\Lip(X,Y)$.

\begin{theorem}\label{theorem:Der-dense-RNP}
Let $Y$ be a Banach space. Then the following are equivalent:
\begin{enumerate}
\item[\emph{(i)}] $Y$ has the Radon-Nikod\'{y}m property.
\item[\emph{(ii)}] $\D(\R,Y)$ is dense in $\Lip(\R,Y)$.
\item[\emph{(iii)}] The set of all Lipschitz maps $f \in \Lip(\R,Y)$ such that
$$
f'(t) = \lim_{h \to 0} \frac{f(t+h)-f(t)}{h}
$$
exists for some $t \in \R$ is dense in $\Lip(\R,Y)$.
\end{enumerate}
\end{theorem}

To prove Theorem~\ref{theorem:Der-dense-RNP}, we need a preliminary lemma to proceed.

\begin{definition}\cite[Definition 5.19]{BenyaLinden}
Let $Y$ be a Banach space, let $I\subset \R$ be an interval, and let $\eps>0$ be given. A function $f\colon  I \longrightarrow  Y$ is said to be \emph{$\eps$-differentiable} at $t_0\in I$ if there are $\delta>0$ and $y \in Y$ such that
$$
\|f(t_0+h)-f(t_0)-hy\| \leq \eps|h|
$$
for every $h \in \R$ with $|h| < \delta$.
\end{definition}

\begin{lemma}\cite[Theorem 5.21]{BenyaLinden} \label{lemma:RNP-diff}
Let $Y$ be a Banach space. Then the following are equivalent:
\begin{itemize}
\item[\emph{(i)}] $Y$ has the Radon-Nikod\'ym property.
\item[\emph{(ii)}] Every Lipschitz map $\psi\colon [0,1] \longrightarrow  Y$ has a point of $\eps$-differentiability for every $\eps>0$.
\end{itemize}
\end{lemma}

\begin{proof}[Proof of Theorem \ref{theorem:Der-dense-RNP}]
(i)$\Rightarrow$(ii). Let $\eps>0$ be given and fix $f\in \Lip(\R,Y)$ with $\|f\|=1$. By Lemma \ref{lemma:RNP-compact-derivable}.(a), there is $\varphi \in L_\infty(\R,Y)$ with $\|\varphi\|_{\infty}=1$ such that
\begin{equation}\label{eq:Der-dense-RNP-representation}
f(t)=\int_0^t \varphi(s)\,ds \qquad \text{for } t\in \R.
\end{equation}
Consider the set
\begin{equation}\label{eq:Der-dens-RNP-set-A}
A_\eps:=\{t\in \R\colon \|\varphi(t)\|>1-\eps\}
\end{equation}
and observe that $A_\eps$ has positive measure. Now,define
$\psi \in L_\infty(\R,Y)$ by
\begin{equation}\label{eq:Der-dens-RNP-def-psi}
\psi(t)=\begin{cases}
\frac{\varphi(t)}{\|\varphi(t)\|} & \text{if } t\in A_\eps, \\ \varphi(t) & \text{otherwise}.
\end{cases}
\end{equation}
It is immediate that $\|\varphi - \psi\|_\infty \leq \eps$. Therefore, defining $g:\R\longrightarrow Y$ by $g(t)=\int_0^t \psi(s)\,ds$ for every $t \in \R$, we obtain that $\|g\|=1$, $\|g-f\|<\eps$, and that $g' = \psi$ a.e.\ by Lemma \ref{lemma:RNP-compact-derivable}.(b).    Since $A_\eps$ has positive measure, there is $t_0\in A_\eps$ such that $g$ is differentiable at $t_0$ and $g'(t_0)=\psi(t_0)$. Further, $\|g'(t_0)\|=\|\psi(t_0)\|=1=\|g\|$, which shows that $g\in \D(\R,Y)$.

(ii)$\Rightarrow$(iii) is clear, because we can consider $t_0 \in \R$ at which $f \in \D(\R,Y)$ attains its norm through a derivative.

(iii)$\Rightarrow$(i). Suppose that $Y$ fails the Radon-Nikod\'ym property. Then, by Lemma \ref{lemma:RNP-diff}, there exist $\eps>0$ and a Lipschitz map $\psi\colon [0,1] \longrightarrow  Y$ such that $\psi$ has no point of $\eps$-differentiability on $[0,1]$. Let $f \in \Lip(\R,Y)$ be defined on $[0,2]$ by
\begin{displaymath}
f(t)=\left\{\begin{array}{@{}cl}
\displaystyle \phantom{.} \psi(t)-\psi(0) & \text{if } 0 \leq t \leq 1, \\
\displaystyle \phantom{.} \psi(2-t)-\psi(0) & \text{if } 1< t \leq 2
\end{array} \right.
\end{displaymath}
and extended $2$-periodic on $\R$. Then, it is easy to see from the definition that $f$ has no point of $\eps$-differentiability on $\R$. That is, for any given $t \in \R$, $\delta>0$ and $y \in Y$, there always exists $h \in \R$ with $|h|<\delta$ such that
\begin{equation} \label{equation:diff}
\|f(t+h)-f(t)-hy\|>\eps |h|.
\end{equation}
Let $g \in \Lip(\R,Y)$ be such that $\|g-f\|<\eps/2$ and
$$
g'(t_0)= \lim_{h\to 0} \frac{g(t_0+h)-g(t_0)}{h}
$$
exists for some $t_0 \in \R$. Choose $\delta>0$ so that
$$
\left\| \frac{g(t_0+h)-g(t_0)}{h} - g'(t_0) \right\| <\frac{\eps}{2}
$$
whenever $|h|<\delta$. By \eqref{equation:diff}, there exists $h \in \R$ with $|h|<\delta$ such that
$$
\|f(t_0+h)-f(t_0)-hg'(t_0)\|>\eps |h|.
$$
Hence we have that
\begin{align*}
\frac{\|(g-f)(t_0+h)-(g-f)(t_0)\|}{|h|} &\geq \left\| \frac{f(t_0+h)-f(t_0)}{h} - g'(t_0) \right\| - \left\| \frac{g(t_0+h)-g(t_0)}{h} - g'(t_0) \right\| \\
&> \eps - \frac{\eps}{2} = \frac{\eps}{2},
\end{align*}
which contradicts the fact that $\|g-f\|<\eps/2$.
\end{proof}

We may also prove the necessity of the Radon-Nikod\'{y}m property of the Banach space $Y$ for the denseness of $\D(X,Y)$ for a nontrivial Banach space $X$. However, we don't know if it can be a sufficient condition and even for $Y=\R$, we don't know any Banach space $X$ with $\dim(X)\geq 2$ such that $\D(X,\R)$ is dense in $\Lip(X,\R)$.

\begin{cor}\label{cor-Der(X,Y)dense=>Y-RNP}
Let $Y$ be a Banach space. If $\D(X,Y)$ is dense in $\Lip(X,Y)$ for a nontrivial Banach space $X$, then $Y$ has the Radon-Nikod\'{y}m property.
\end{cor}

\begin{proof}
Suppose that $Y$ fails the Radon-Nikod\'{y}m property. Define a Lipschitz map $f_0\in \Lip(\R,Y)$ with $\|f_0\|=1$ as it is done in the proof of (iii)$\Rightarrow$(i), Theorem \ref{theorem:Der-dense-RNP}. That is, there exists $0<\eps<1$ such that for any given $t\in \R$, $\delta>0$ and $y\in Y$, there exists $h\in \R$ with $|h|<\delta$ such that
\begin{equation} \label{equation:diff-2}
\|f_0(t+h)-f_0(t)-hy\|>\eps |h|.
\end{equation}
Pick $x^*\in \NA(X,\R)$ with $\|x^*\|=1$. Define $f\colon X\longrightarrow Y$ by
$$
f(x)=f_0(x^*(x)) \qquad \text{for } x\in X.
$$
It is routine to show that $f\in \Lip(X,Y)$ and that $\|f\|=\|f_0\|=1$. Suppose now that we can find $g\in\D(X,Y)$ with $\|g\|=1$ and $\|f - g\|<\eps/2$. By definition of $\D(X,Y)$, there are $x_0\in X$ and $u\in S_X$ such that $g'(x_0,u)$ exists and belongs to $S_Y$. Therefore, there exists $\delta>0$ such that
\begin{equation}\label{eq:coroDERdenseRNP-2}
\|g(x_0+tu)-g(x_0)-tg'(x_0,u)\|<\frac{\eps|t|}{2},~~ \text{whenever }~ |t|< \delta.
\end{equation}
We now define $f_1, g_1:\R\longrightarrow Y$ by
\begin{equation*}
f_1(t) = f(x_0+tu)= f_0(x^*(x_0) + tx^*(u)), \quad g_1(t)=g(x_0+tu)~ \text{ for }~t\in \R,
\end{equation*}
and observe that
\begin{equation}\label{eq:coroDERdenseRNP-3}
\|f_1 - g_1\|\leq \|f-g\|< \frac{\eps}{2}.
\end{equation}
We now have two possibilities:

\noindent (a). If $x^*(u)=0$, then $f_1$ is constant on $\R$, so $f_1'\equiv 0$. This contradicts the facts that $g_1'(0)=g'(x_0,u)\in S_Y$ and that $\|f_1-g_1\|<\eps/2<1$.

\noindent (b). If $x^*(u)\neq 0$, then it follows from \eqref{eq:coroDERdenseRNP-2} and \eqref{eq:coroDERdenseRNP-3} that
$$
\left\|f_0\bigl(x^*(x_0)+tx^*(u)\bigr) - f_0\bigl(x^*(x_0)\bigr) - tx^*(u)\frac{g'(x_0,u)}{x^*(u)}\right\| <\eps|t|~~  \text{for }~ |t|<\delta.
$$
But this enters into a contradiction with \eqref{equation:diff-2}.
\end{proof}

If one deals with Lipschitz compact maps, then the proof of (i)$\Rightarrow$(ii) of Theorem \ref{theorem:Der-dense-RNP} can be repeated without any assumption on $Y$, getting the second main result of this section.

\begin{theorem}\label{theorem:Dkalwaysdense}
Let $Y$ be a Banach space. Then $\Dk(\R,Y)$ is dense in $\Lipk(\R,Y)$.
\end{theorem}

\begin{proof}
We just repeat the proof of (i)$\Rightarrow$(ii) in Theorem \ref{theorem:Der-dense-RNP}, taking into account that if $f\in \Lipk(\R,Y)$, then it is representable by an integral as in \eqref{eq:Der-dense-RNP-representation} by Lemma \ref{lemma:RNP-compact-derivable}.(a). Finally, the Lipschitz map $g$ obtained in the proof is Lipschitz compact by Lemma \ref{lemma:RNP-compact-derivable}.(b), because $\psi$ defined in \eqref{eq:Der-dens-RNP-def-psi} has a relatively compact range a.e.\ from $f\in \Lipk(\R.Y)$.
\end{proof}

As we already commented, it is not true that $A(X,Y)$ is always dense in $\Lip(X,Y)$ \cite{Gode2} (see \cite[p.~109]{Gode1}). For the sake of completeness, we include here a short justification of this example.

\begin{example}[\textrm{\cite{Gode2}}]\label{example:Godefroy-negative_A}
{\slshape Let $X=c_0$ and let $Y$ be an equivalent renorming of $c_0$ with the Kadec-Klee property} ({\slshape which it is well-known that exists}){\slshape. Then, $\A(X,Y)$ is not dense in $\Lip(X,Y)$.}

In fact, it is shown in \cite[Corollary 3.5]{Gode2} that no Lipschitz isomorphism from $X$ onto $Y$ belongs to $\A(X,Y)$ (as $c_0$ is asymptotically uniformly flat, see \cite[Definition 2.1]{Gode2}). As the identity map belongs to the set of Lipschitz isomorphism from $X$ onto $Y$, it is enough to prove that this set is open. This is surely well-known to experts, but we would like to include an easy argument which has been given to us by G.~Lancien. Assume that $f$ is a Lipschitz isomorphism from $X$ onto $Y$. Consider $g \in \Lip(X,Y)$. If $\|g\|$ is small enough, then $f-g$ is trivially bi-Lipschitz from $X$ onto its image. Thus, the only thing we have to show is that $f-g$ is surjective by solving the equation $y=f(x)-g(x)$ for any $y \in Y$. This is possible again, provided that $\|g\|$ is small enough, applying the Banach fixed point theorem on the contraction map $f^{-1}\bigl(y+g(\cdot)\bigr)\colon  X \longrightarrow  X$.
\end{example}

In the last part of this section, we show some other results on the denseness of norm attaining Lipschitz maps.

\begin{prop}\label{prop:axy2}
Let $X$ be any Banach space and let $Y$ be a uniformly convex Banach space. If $T_f \in \NA(\F(X),Y)$, then $f \in \A(X,Y)$.
\end{prop}

It is easy to see that the reversed result to the above one does not hold: just consider $Y=\R$ and an arbitrary Banach space $X$. Then $A(X,\R)=\Lip(X,\R)$ by Proposition~\ref{prop:inclusion}.(d), but $\NA(\F(X),\R)\neq \Lin(\F(X),\R)$, because $\F(X)$ is not reflexive as it contains $\ell_1$. To prove Proposition~\ref{prop:axy2}, we need the following result from \cite{AcoBerGarKimMae1}. Recall that if $X$ is a Banach space, for given  $x^*\in S_{X^*}$ and $\delta > 0$, the corresponding \emph{slice} of $B_X$ is defined as  $$\operatorname{S}(B_X,x^*,\delta):=\{x\in B_X \colon  x^*(x)> 1-\delta\}.$$

\begin{lemma}[\mbox{\rm \cite[Lemma 2.1]{AcoBerGarKimMae1}}]\label{lem:axy1}
Let $Y$ be a uniformly convex Banach space. Then, for every $\eps > 0$ and $y^* \in S_{Y^*}$, we have
$$
\diam \bigl(\operatorname{S}(B_Y,y^*,\delta_Y(\eps)) \bigr)\leq \eps.
$$
\end{lemma}

\begin{proof}[Proof of Proposition \ref{prop:axy2}]
Let $\| T_f(w) \| = \| T_f \|$ for $w \in S_{\F(X)}$. Since $B_{\F(X)} = \overline{\co}(\Mol(X))$ by Lemma \ref{Lemm:elementarypropertiesF(X)}.(c), there exists a sequence $\{w_n\}_{n=1}^{\infty} \subset \co(\Mol(X))$ converging to $w$. By the uniform convexity of $Y$, we can find a sequence $\{ u_n \}_{n=1}^{\infty} \subset \Mol(X)$ such that
$$
\| T_f(u_n)-T_f(w) \| \leq \frac{1}{n} \quad \text{for each } n \in \N.
$$
Indeed, assume $\| T_f(w) \| = 1$ and let $y^* \in S_{Y^*}$ be such that $y^*(T_f(w)) = 1$. Suppose that $\| T_f(w_n) - T_f(w) \| < \delta_Y(n^{-1})$, where $\delta_Y$ is the modulus of convexity of $Y$ and $w_n=\sum_{j=1}^{k(n)} \alpha_j v_{n,j}$ is a convex combination of $v_{n,j} \in \Mol(X)$ for $1 \leq j \leq k(n)$. Then, there exists some $j$ such that $T_f(v_{n,j}) \in \operatorname{S}(B_Y,y^*,\delta_Y(n^{-1}))$, because $y^*(T_f(w_n)) > 1 - \delta_Y(n^{-1})$. By Lemma \ref{lem:axy1}, we have
$$
\diam \bigl(\operatorname{S}(B_Y,y^*,\delta_Y(n^{-1}))\bigr) \leq \frac{1}{n},
$$
which implies that
$$
\| T_f(u_n)-T_f(w) \| \leq \frac{1}{n}
$$
if we let $u_n=v_{n,j}$. Note that each $u_n$ is of the form $\frac{\delta(x_n)-\delta(y_n)}{\| x_n-y_n \|} \in M$, thus we can deduce that
\[
\frac{f(x_n)-f(y_n)}{\| x_n-y_n \|} \longrightarrow T_f(w) \quad \text{with } \| T_f(w) \| = \| T_f \| = \| f \|.\qedhere
\]
\end{proof}

The following results are straightforward consequences of Proposition \ref{prop:axy2}.

\begin{cor}\label{cor:aunifcvx}
Let $X$ and $Y$ be Banach spaces such that $Y$ is uniformly convex and $\NA(\F(X),Y)$ is dense in $\Lin(\F(X),Y)$. Then, $\A(X,Y)$ is dense in $\Lip(X,Y)$.
\end{cor}

\begin{cor}\label{cor:daunifcvx}
Let $X$ be a finite-dimensional Banach space and let $Y$ be an uniformly convex Banach space. Suppose that $\NA(\F(X),Y)$ is dense in $\Lin(\F(X),Y)$. Then, $\DA(X,Y)$ is dense in $\Lip(X,Y)$.
\end{cor}

%%%%%%%%%%%%%%%%%%%%%%%%%%%%%%%%%%%%%%%%%%%%%%%%%%%%%%%%%%%%%%%%
\section{Local Directional Bishop-Phelps-Bollob\'{a}s property for Lipschitz Maps} \label{chapter:dlbpbp}
%%%%%%%%%%%%%%%%%%%%%%%%%%%%%%%%%%%%%%%%%%%%%%%%%%%%%%%%%%%%%%%%

We would like to deal now with the $\LDBPBp$, trying to extend some results of \cite{KadMarSol} from the scalar-valued case to the vector-valued case. Our main result in this section is the following.

\begin{theorem}\label{theorem:unifcvx}
Let $X$ and $Y$ be Banach spaces such that $X$ is uniformly convex and $(\F(X),Y)$ has the $\BPBp$ for compact operators. Then, the pair $(X,Y)$ has the $\LDBPBp$ for Lipschitz compact maps.

In fact, we have something more: for every $\eps>0$, there exists $\eta >0$ such that for any positive function $\rho \colon   \widetilde{X} \longrightarrow \R$ and whenever $f \in S_{\Lipk(X,Y)}$  and $(x,y) \in \widetilde{X}$ satisfy
$$
\frac{\| f(x)-f(y) \|}{\| x-y \| } > 1 - \eta,
$$
there exist $g \in S_{\Lipk(X,Y)}$, $z \in S_Y$, $u \in S_X$ and $\bar{x} \in X$ such that $g$ attains its norm locally directionally at the point $\bar{x}$ in the direction $u$ toward $z$, $\| g -f \| < \eps$, $\bigl\| u - \frac{x-y}{\| x - y \|} \bigr\| < \eps$ and $\dist (\bar{x}, [x,y] ) < \eps \rho (x,y)$.

\end{theorem}

To give a proof of Theorem~\ref{theorem:unifcvx}, we need the following lemmas which generalize some of the results of \cite{KadMarSol}. We state the proofs, because there are some significant differences with the original ones.

\begin{lemma}\label{lem:LDA1}
Let $X$ and $Y$ be Banach spaces such that $(\F(X),Y)$ has the $\BPBp$ for compact operators witnessed by the function $\eps \mapsto \eta(\eps)$. Suppose $0<\eps<1$, $f \in S_{\Lipk(X,Y)}$ and $(x,y) \in \widetilde{X}$ satisfy
$$
\frac{\| f(x)-f(y) \|}{\| x-y \|} > 1 - \eta(\eps).
$$
Then, for every $h \in S_{\Lip(X,Y)}$ satisfying
$$
\frac{\| h(x)-h(y) \|}{\| x-y \|} = 1,
$$
there exist $g \in S_{\Lipk(X,Y)}$, $z \in S_Y$ and $\{ (v_n,w_n) \}_{n=1}^{\infty} \subset \widetilde{X}$ such that
$$
\| g - f \| < \eps, \quad \frac{\| h(v_n)-h(w_n) \|}{\| v_n-w_n \|} > 1 - \eps \quad \text{for every } n \in \N \quad \text{and} \quad \frac{g(v_n)-g(w_n)}{\| v_n-w_n \|} \longrightarrow z.
$$
\end{lemma}

\begin{proof}
Consider $T_f \in S_{\K(\F(X),Y)}$ and let $m = \frac{\delta(x)-\delta(y)}{\| x-y \|} \in S_{\F(X)}$. Since $\| T_f(m) \|> 1 - \eta(\eps)$ and $(\F(X),Y)$ has the $\BPBp$ for compact operators witnessed by the function $\eps \mapsto \eta(\eps)$, we can find $T_g \in S_{\K(\F(X),Y)}$ and $w \in S_{\F(X)}$ such that
$$
\| T_g(w) \| = 1, \quad \| T_g - T_f \| < \eps \quad \text{and} \quad \| w - m \| < \eps.
$$
Since $B_{\F(X)} = \overline{\co}(\Mol(X))$ by Lemma \ref{Lemm:elementarypropertiesF(X)}.(c), we can find $0<\nu < \eps$ and a sequence $\{w_n\}_{n=1}^{\infty} \subset \co(\Mol(X))$ such that
$$
\| w_n - m \| < \nu \quad \text{and} \quad \| T_g(w_n) \| > 1- \frac{1}{n^2}.
$$
Now, for $T_h \in S_{\Lin(\F(X),Y)}$ we have
$$
\| T_h(w_n) \| \geq \| T_h(m) \| - \| m - w_n \| > 1 - \nu.
$$
Thus
$$
\frac{1}{n} \| T_h(w_n) \| + \left( 1 - \frac{1}{n} \right) \| T_g(w_n) \| > \frac{1}{n} (1- \nu) + \left( 1- \frac{1}{n} \right) \left( 1- \frac{1}{n^2}\right).
$$
Therefore, for every $n\in \N$ we can find  $u_n \in \Mol(X)$ so that
$$
\frac{1}{n} \| T_h(u_n) \| + \left( 1 - \frac{1}{n} \right) \| T_g(u_n) \| > \frac{1}{n} (1- \nu) + \left( 1- \frac{1}{n} \right) \left( 1- \frac{1}{n^2}\right).
$$
Hence, from the facts that $\| T_h(u_n) \| \leq 1$ and $\| T_g(u_n) \| \leq 1$, we get routinely that
$$
\|T_h(u_n)\| > 1- \nu -\frac{1}{n} \left( 1 -\frac{1}{n} \right) \quad \text{and} \quad \|T_g(u_n)\| > 1- \frac{1}{n^2} - \frac{\nu}{n-1}.
$$
So we may assume that
$$
\| T_h(u_n) \| > 1 - \eps \quad \text{and} \quad \| T_g(u_n) \| \longrightarrow 1
$$
passing to a subsequence, if necessary. Note that each $u_n$ is of the form $\frac{\delta(v_n)-\delta(w_n)}{\| v_n-w_n \|}$ for suitable $(v_n,w_n)\in \widetilde{X}$. By compactness of $T_g$, there exists $z \in S_Y$ such that
$$
T_g(u_n)=\frac{g(v_n)-g(w_n)}{\|v_n-w_n\|} \longrightarrow z
$$ passing to a subsequence, if necessary. Finally, $g \in \Lipk(X,Y)$ by Lemma~\ref{Lemm:elementarypropertiesF(X)}.(b), so we have obtained the desired result.
\end{proof}

\begin{lemma}\label{lem:LDA2}
Let $X$ and $Y$ be Banach spaces such that $X$ is uniformly convex and $(\F(X),Y)$ has the $\BPBp$ for compact operators. Then, for every $\eps>0$ there exists $\eta >0$ such that for any positive function $\rho \colon   \widetilde{X} \longrightarrow  \R$ and whenever $f \in S_{\Lipk(X,Y)}$ and $(x,y) \in \widetilde{X}$ satisfy
$$
\frac{\| f(x)-f(y) \|}{\| x-y \| } > 1 - \eta,
$$
there exist $g \in S_{\Lipk(X,Y)}$, $z \in S_Y$ and $\{ (v_n,w_n) \}_{n=1}^{\infty} \subset \widetilde{X}$ such that
$$
\frac{g(v_n)-g(w_n)}{\| v_n - w_n \|} \longrightarrow z, \quad \| g -f \| < \eps, \quad \left\| \frac{x-y}{\| x - y \|} - \frac{v_n-w_n}{\| v_n-w_n \|} \right\| < \eps,
$$
and $\| v_n-w_n \| < \eps \rho (x,y)$, $\dist (v_n, [x,y] ) < \eps \rho (x,y)$.
\end{lemma}

\begin{proof}
Assume that $(\F(X),Y)$ has the $\BPBp$ for compact operators witnessed by the function $\eps \mapsto \eta_0(\eps)$ with $\eta_0(\eps) <\eps$. Let $0<\eps<1/4$ and put $\eta := \eta_0 \bigl( \min \bigl\{ \eps, \delta_X(\eps)/2 \bigr\} \bigr) > 0$, where $\delta_X$ is the modulus of convexity of $X$. Suppose that a positive function $\rho\colon   \widetilde{X} \longrightarrow  \R$ is given and $f \in S_{\Lipk(X,Y)}$ and $(x,y) \in \widetilde{X}$ satisfy
$$
\frac{\| f(x)-f(y) \|}{\| x-y \| } > 1 - \eta.
$$
Choose $\tilde{x},\tilde{y} \in \widetilde{[x,y]}$ such that
$$
\frac{\| f(\tilde{x})-f(\tilde{y}) \|}{\| \tilde{x}-\tilde{y} \| } > 1 - \eta, \quad \frac{\tilde{x} - \tilde{y}}{\| \tilde{x} - \tilde{y} \|} = \frac{x - y}{\| x - y \|} \quad \text{and} \quad \| \tilde{x} - \tilde{y} \| < \frac{1}{4} \min \bigl\{ \eps \rho (x,y), \| \tilde{x} \| , \| \tilde{y} \| \bigr\}.
$$
Fix any $y_0 \in S_Y$ and define $F \in \Lip(X,Y)$ by $F(w) := \max \{ \| \tilde{x} - \tilde{y} \| - \| \tilde{x} - w \| , 0 \}\,y_0$.
Then, $\| F \| = 1$ and
$$
\frac{\| F(\tilde{x}) - F(\tilde{y}) \|}{\| \tilde{x} - \tilde{y} \|} = 1.
$$
Let $x^* \in S_{X^*}$ be such that
$$
x^* \left(\frac{\tilde{x} - \tilde{y}}{\| \tilde{x} - \tilde{y} \|}\right) = 1.
$$
If we define $h := \frac{1}{2} (F+y_0 x^*)$, then $h \in S_{\Lip(X,Y)}$ and
$$
\frac{\| h(\tilde{x}) - h(\tilde{y}) \|}{\| \tilde{x} -\tilde{y} \|} = 1.
$$
Applying Lemma \ref{lem:LDA1}, we can find $g \in S_{\Lipk(X,Y)}$, $z \in S_Y$ and $\{ (v_n,w_n) \}_{n=1}^{\infty} \subset \widetilde{X}$ such that
$$
\| f - g \| < \eps, \quad \frac{\| h(v_n)-h(w_n) \|}{\| v_n-w_n \|} > 1- \min \left\{ \eps, \frac{\delta_X(\eps)}{2} \right\} \quad \text{and} \quad \frac{g(v_n)-g(w_n)}{\| v_n - w_n \|} \longrightarrow z.
$$
From the second inequality, we get
\begin{equation}\label{eq:lemma}
\frac{\| F(v_n)-F(w_n) \|}{\| v_n-w_n \|} > 1-2\eps \quad \text{and} \quad \frac{|x^*(v_n)-x^*(w_n)|}{\| v_n-w_n \|} > 1- \delta_X(\eps).
\end{equation}
Here, we may assume $x^*(v_n)-x^*(w_n) >0$ replacing $z$ by $-z$ if necessary. Since
$$
x^*\left(\frac{x-y}{\| x-y \|}\right) = x^*\left(\frac{\tilde{x}-\tilde{y}}{\| \tilde{x}-\tilde{y} \|}\right) = 1 > 1 - \delta_X(\eps),
$$
we obtain by the uniform convexity of $X$ that
$$
\left\| \frac{x-y}{\| x - y \|} - \frac{v_n-w_n}{\| v_n-w_n \|} \right\| < \eps.
$$
Now,
$$
\| v_n - w_n \| < \frac{\| F(v_n) - F(w_n) \|}{1- 2 \eps} \leq \frac{\eps \rho (x,y)}{4(1- 2 \eps )} < \frac{1}{2}\eps \rho (x,y).
$$
Suppose that $v_n \in \supp F$. Then, combined with the fact that $\|\tilde{x}-\tilde{y}\|\leq  \frac{1}{4}\eps \rho (x,y)$, we can deduce that
$$\dist(v_n, [x,y])< \frac{1}{4}\eps \rho (x,y) < \eps \rho (x,y).$$

If $v_n\notin \supp F$, then we must have that $w_n \in \supp F$ by \eqref{eq:lemma}. Hence, we have
$$\dist(v_n, [x,y])\leq \|v_n-w_n\| + \dist(w_n, [x,y]) < \rho(x,y),$$
which completes the proof.

\end{proof}

We are now ready to present the pending proof.

\begin{proof}[Proof of Theorem \ref{theorem:unifcvx}]
Let $\eps > 0$ be given. Set $\eps_n := \frac{\eps}{2^{n+1}}$ for each $n \in \N$ and $\eta := \eta_1 (\eps_1)$ where $\eta_1$ is from Lemma \ref{lem:LDA2}. Suppose that a positive function $\rho\colon  \widetilde{X} \longrightarrow  \R$ is given and $f \in S_{\Lipk(X,Y)}$ and $(x,y) \in \widetilde{X}$ satisfy
$$
\frac{\| f(x)-f(y) \|}{\| x-y \| } > 1 - \eta.
$$
Consider $\tau:=\min \{ 1, \rho \}$. We use Lemma \ref{lem:LDA2} to get $f_2 \in S_{\Lipk(X,Y)}$, $z_2 \in S_Y$ and $\{ (v_n,w_n) \}_{n=1}^{\infty} \subset \widetilde{X}$ satisfying the conditions given there with $\eps_1$ and $\tau$.
Choose $n_1 \in \N$ such that
$$
\left\| \frac{f_2 (v_{n_1}) - f_2 (w_{n_1})}{\| v_{n_1} - w_{n_1} \|} - z_2 \right\| <  \eta_1 \bigl(\eps_2 \tau (x,y)\bigr).
$$
Set $x_2 := v_{n_1}$, $y_2 := w_{n_1}$, $f_1 := f$, $x_1 := x$ and $ y_1 := y$. So far we have
\begin{enumerate}[(a)]
\item $\displaystyle \left\| \frac{f_2 (x_2) - f_2 (y_2)}{\| x_2 - y_2 \|} - z_2 \right\| < \eta_1 \bigl( \eps_2 \tau (x_1,y_1)\bigr),$ \\
\item $\displaystyle \| f_1 - f_2 \| < \eps_1,$ \\
\item $\displaystyle \left\| \frac{x_1 - y_1}{\| x_1 - y_1 \|} - \frac{x_2 - y_2}{\| x_2 - y_2 \|} \right\| < \eps_1,$ \\
\item $\displaystyle \| x_2 - y_2 \| < \eps_1 \tau (x_1,y_1),$ \\
\item $\displaystyle \dist ( x_2, [x_1,y_1] ) < \eps_1 \tau (x_1,y_1).$ \\
\end{enumerate}
Now, by an inductive procedure, we get $\{ f_n \}_{n=1}^\infty \subset S_{\Lipk(X,Y)}$, $\{z_n\}_{n=1}^\infty \subset S_Y$ and $\{ ( x_n , y_n ) \}_{n=1}^{\infty} \subset \widetilde{X}$ satisfying for every $n \in \N$:
\begin{enumerate}[(a')]
\item $\displaystyle \left\| \frac{f_{n+1} (x_{n+1}) - f_{n+1} (y_{n+1})}{\| x_{n+1} - y_{n+1} \|} - z_{n+1} \right\| < \eta_1 \bigl( \eps_{n+1} \tau (x_1,y_1)\bigr),$ \\
\item $\displaystyle \| f_n - f_{n+1} \| < \eps_n \tau (x_1,y_1) \leq \eps_n,$ \\
\item $\displaystyle \left\| \frac{x_n - y_n}{\| x_n - y_n \|} - \frac{x_{n+1} - y_{n+1}}{\| x_{n+1} - y_{n+1} \|} \right\| < \eps_n \tau (x_1,y_1) \leq \eps_n,$ \\
\item $\displaystyle \| x_{n+1} - y_{n+1} \| < \eps_n \tau (x_1,y_1) \cdot \tau (x_n,y_n) \leq \eps_n \tau (x_1,y_1),$ \\
\item $\displaystyle \dist ( x_{n+1}, [x_n,y_n] ) < \eps_n \tau (x_1,y_1) \cdot \tau (x_n,y_n) \leq \eps_n \tau (x_1,y_1).$ \\
\end{enumerate}
From (d') and (e'), we have that
$$
\| x_{n+1} - x_{n+2} \| < \eps_n \tau (x_1,y_1) + \eps_{n+1} \tau (x_1,y_1) \leq \eps_n + \eps_{n+1} \longrightarrow 0
$$
as $n$ tends to $\infty$, so that there exists $\bar{x} \in X$ such that both $x_n, y_n$ converge to $\bar{x}$ with the aid of (d').
Note that
\begin{align*}
\dist (\bar{x}, [x,y] ) & \leq \dist (x_2, [x_1,y_1] ) + \| x_2 -\bar{x} \| \\
&\leq \dist (x_2, [x_1,y_1] )  + \sum_{n=1}^\infty \| x_{n+1} - x_{n+2} \| \\
&< 2 \sum_{n=1}^\infty \eps_n \tau (x_1,y_1) \leq \eps \rho (x,y).
\end{align*}
From (b') and (c'), there exist $g \in S_{\Lipk(X,Y)}$ and $u \in S_X$ such that
$$
f_n \longrightarrow g \quad \text{and} \quad \frac{x_n - y_n}{\| x_n - y_n \|} \longrightarrow u.
$$
Then,
$$
\| f - g \| \leq \sum_{n=1}^\infty \eps_n < \eps \quad \text{and} \quad \left\| \frac{x-y}{\| x-y \|} - u \right\| \leq \sum_{n=1}^\infty \eps_n < \eps.
$$
From (a'), we obtain that
\begin{align*}
\frac{ \| g(x_n)-g(y_n) \|}{\| x_n-y_n \|} &\geq \frac{ \| f_n(x_n)-f_n(y_n) \|}{\| x_n-y_n \|} - \frac{ \| (g-f_n)(x_n)-(g-f_n)(y_n) \|}{\| x_n-y_n \|} \\
&> 1 -  \eta_1 \bigl( \eps_n \tau (x_1,y_1)\bigr) - \frac{ \| (g-f_n)(x_n)-(g-f_n)(y_n) \|}{\| x_n-y_n \|} \longrightarrow 1
\end{align*}
as $n$ tends to $\infty$. Thus by compactness, we may conclude that $\frac{g(x_n)-g(y_n)}{\| x_n-y_n \|} \longrightarrow  z$ for some $z \in S_Y$ passing to a subsequence if necessary.
\end{proof}

As a direct consequence of Theorem \ref{theorem:unifcvx} and \cite[Example 1.5]{DanGarMaeMar}, we get the following certain result.

\begin{cor}
\label{cor:universalrange}
Let $X$ be a uniformly convex Banach space and let $Y$ be a Banach space satisfying one of the following properties:
\begin{enumerate}
\item[\textup{(a)}] $Y$ has property $\beta$,
\item[\textup{(b)}] $Y^*$ is isometrically isomorphic to some $L_1(\mu)$-space.
\end{enumerate}
Then, the pair $(X,Y)$ has the $\LDBPBp$ for Lipschitz compact maps. In particular, $\LDAk(X,Y)$ is dense in $\Lipk(X,Y)$.
\end{cor}

The next result is a slightly different version of Theorem \ref{theorem:unifcvx} when the domain space is specified to a Hilbert space.

\begin{theorem}\label{theorem:hilbert}
Let $H$ be a Hilbert space and let $Y$ be a Banach space. Suppose that $(\F(H),Y)$ has the $\BPBp$ for compact operators. Then, for every $\eps>0$, there exists $\eta >0$ such that whenever $f \in S_{\Lipk(X,Y)}$ and $(x,y) \in \widetilde{X}$ satisfy
$$
\frac{\| f(x)-f(y) \|}{\| x-y \| } > 1 - \eta,
$$
there exist $g \in S_{\Lipk(X,Y)}$, $z \in S_Y$ and $\bar{x} \in X$ such that $g$ attains its norm locally directionally at the point $\bar{x}$ in the direction $\frac{x-y}{\| x-y \|}$ toward $z$, $\| g -f \| < \eps$ and $\dist (\bar{x}, [x,y] ) < \eps\max\{\|x\|,\|y\|\}$.
\end{theorem}

\begin{proof}
Let $0<\eps<1$ be given. Choose $\eta > 0$ as in Theorem \ref{theorem:unifcvx} applied with $\frac{\eps}{3}$ and $\rho(x,y)=\max\{\|x\|,\|y\|\}$. Suppose $f \in S_{\Lipk(H,Y)}$ and $(x,y) \in \widetilde{H}$ satisfy
$$
\frac{\| f(x)-f(y) \|}{\| x-y \| } > 1 - \eta.
$$
Then, we can find $\tilde{g} \in S_{\Lipk(H,Y)}$, $z \in S_Y$, $u \in S_H$, $\tilde{x} \in H$ and $ \{ (\tilde{x}_n,\tilde{y}_n) \}_{n=1}^{\infty} \subset \widetilde{H}$ such that
$$
\tilde{x}_n, \tilde{y}_n \longrightarrow \tilde{x}, \quad \frac{\tilde{x}_n-\tilde{y}_n}{\| \tilde{x}_n-\tilde{y}_n \|} \longrightarrow u, \quad \frac{\tilde{g}(\tilde{x}_n) - \tilde{g}(\tilde{y}_n)}{\| \tilde{x}_n-\tilde{y}_n \|} \longrightarrow z,
$$
and
$$
\| \tilde{g} - f \| < \frac {\eps}{3}, \quad \left\| \frac{x-y}{\| x-y \|} - u \right\| < \frac{\eps}{3}, \quad \dist ( \tilde{x}, [x,y]) < \frac{\eps}{3}\rho(x,y).
$$
Since $H$ is a Hilbert space, there exists a linear isometry $R\colon  H \longrightarrow  H$ such that
$$
R(u)=\frac{x-y}{\| x-y \|}~~ \text{and} ~~\| R - \Id_H \| < \frac{\eps}{3}
$$
(see \cite[Lemma 2.2]{AcoMasSol} for instance). Now, consider $g := \tilde{g} \circ R^{-1} \in S_{\Lipk(H,Y)}$, $\{(x_n,y_n) \}_{n=1}^{\infty} \subset \widetilde{H}$ with $(x_n,y_n) = \bigl(R(\tilde{x}_n),  R(\tilde{y}_n)\bigr)$ for each $n \in \N$, and put $\bar{x}:= R(\tilde{x})$. Then,
\begin{enumerate}[(a)]
\item $\displaystyle x_n,y_n \longrightarrow R(\tilde{x})=\bar{x}$,
\item $\displaystyle \frac{x_n-y_n}{\| x_n-y_n \|} = R\left(\frac{\tilde{x}_n-\tilde{y}_n}{\| \tilde{x}_n-\tilde{y}_n \|}\right) \longrightarrow R(u) = \frac{x-y}{\| x-y \|},$ \\
\item $\displaystyle \frac{g(x_n)-g(y_n)}{\| x_n-y_n \|} = \frac{\tilde{g}(\tilde{x}_n) - \tilde{g}(\tilde{y}_n)}{\| \tilde{x}_n-\tilde{y}_n \|} \longrightarrow z,$ \\
\item $\displaystyle \| g- f \| \leq \| g - \tilde{g} \| + \| \tilde{g} - f \| \leq \| \tilde{g} \circ R^{-1} \| \| R- \Id_H \| + \| \tilde{g} - f \| < \frac{\eps}{3} + \frac{\eps}{3} < \eps$,\\
\item $\displaystyle \dist (\bar{x}, [x,y] ) \leq \| R(\tilde{x}) - \tilde{x} \| + \dist ( \tilde{x}, [x,y] )\\ \\
\phantom{-..WWW-} \leq \| R- \Id_H \| \| \tilde{x} \| + \dist ( \tilde{x}, [x,y] ) \\ \\
\phantom{-..WWW-} < \frac{\eps}{3} \bigl( \max \{ \| x \| , \| y \| \} + \frac{\eps}{3} \rho (x,y) \bigr) + \frac{\eps}{3} \rho (x,y) < \eps \rho (x,y).$ \qedhere
\end{enumerate}
\end{proof}

Again from \cite[Example 1.5]{DanGarMaeMar}, we can derive the same type of results as those given in Corollary~\ref{cor:universalrange}.

\begin{cor}
Let $H$ be a Hilbert space and let $Y$ be a Banach space satisfying one of the following properties:
\begin{enumerate}
\item[\textup{(a)}] $Y$ has property $\beta$,
\item[\textup{(b)}] $Y^*$ is isometrically isomorphic to some $L_1(\mu)$-space.
\end{enumerate}
Then, for every $\eps>0$, there exists $\eta >0$ such that whenever $f \in S_{\Lipk(X,Y)}$ and $(x,y) \in \widetilde{X}$ satisfy
$$
\frac{\| f(x)-f(y) \|}{\| x-y \| } > 1 - \eta,
$$
there exist $g \in S_{\Lipk(X,Y)}$, $z \in S_Y$ and $\bar{x} \in X$ such that $g$ attains its norm locally directionally at the point $\bar{x}$ in the direction $\frac{x-y}{\| x-y \|}$ toward $z$, $\| g -f \| < \eps$ and $\dist (\bar{x}, [x,y] ) < \eps\max\{\|x\|,\|y\|\}$.
\end{cor}

As a special case of Theorems \ref{theorem:unifcvx} and \ref{theorem:hilbert}, we may obtain a strengthened result when $X=\R$. In this case, $\F(\R) = L_1(\R)$ by Lemma~\ref{Lemm:elementarypropertiesF(X)}.(d) and then, the pair $(\F(\R),Y)$ has the BPBp for compact operators if and only if $Y$ has a property called $\AHSP$ introduced in \cite[Definition 3.1]{AcoAroGarMae} (see \cite[Remark 2.5]{AcoBerGarKimMae2} for the result). In fact, we may squeeze the ideas developed in Theorems \ref{theorem:Der-dense-RNP} and \ref{theorem:Dkalwaysdense} to get much better results.

\begin{prop}\label{prop:BPBp-Der-RNP-slopes}
Let $Y$ be a Banach space with the Radon-Nikod\'{y}m property. Then, for every $\eps>0$, if $f\in S_{\Lip(\R,Y)}$ and
$(t_1, t_2)\in \widetilde{\R},~ t_1<t_2$ satisfy that
$$
\frac{\|f(t_1)-f(t_2)\|}{|t_1-t_2|}>1-\eps,
$$
 then there exists $g\in S_{\Lip(\R,Y)}$ and $t_0\in [t_1, t_2]$ such that $g$ is differentiable at $t_0$ with $\|g'(t_0)\|=1$ and $\|f-g\|<\eps$.
\end{prop}

\begin{proof}
By Lemma \ref{lemma:RNP-compact-derivable}.(a), $f$ is differentiable a.e.\ and we may write
$$
f(t)=\int_{0}^{t}f'(s)\,ds ~~ \text{for } ~ t\in \R.
$$
By hypothesis,
$$
1-\eps<\frac{\|f(t_1)-f(t_2)\|}{|t_1-t_2|}= \frac{1}{t_2-t_1}\left\|\int_{t_1}^{t_2}f'(s)\,ds\right\| \leq  \frac{1}{t_2-t_1}\int_{t_1}^{t_2}\|f'(s)\|\,ds,
$$
so we get that the set
$$
\widetilde{A}_\eps:=\{t\in[t_1,t_2]\colon f'(t) ~\mbox{exists and}~ \|f'(t)\| >1-\eps\}
$$
has positive measure. Now, we may follow the proof of (i)$\Rightarrow$(ii) of Theorem \ref{theorem:Der-dense-RNP}, applying the set $\widetilde{A}_\eps$ instead of the set $A_\eps$ defined in \eqref{eq:Der-dens-RNP-set-A}, to obtain $g\in \Lip(\R,Y)$ and $t_0\in \widetilde{A}_\eps\subset [t_1, t_2]$ such that $\|f-g\|< \eps$, $g$ is differentiable at $t_0$ and $\|g\|=\|g'(t_0)\|=1$.
\end{proof}

For a Lipschitz compact map, we have a more general result. The proof is an obvious adaptation of the previous one using the ideas of the proof of Theorem \ref{theorem:Dkalwaysdense}.

\begin{prop}\label{prop:BPBpk-Der-all-slopes}
Let $Y$ be a Banach space. Then, for every $\eps>0$, if $f\in S_{\Lipk(\R,Y)}$ and $(t_1, t_2)\in \widetilde{\R}, ~t_1<t_2$ satisfy that
$$
\frac{\|f(t_1)-f(t_2)\|}{|t_1-t_2|}>1-\eps,
$$
then there exist $g\in S_{\Lipk(\R,Y)}$ and $t_0\in [t_1, t_2]$ such that $g$ is differentiable at $t_0$ with $\|g'(t_0)\|=1$ and $\|f-g\|<\eps$.
\end{prop}

Note that Proposition \ref{prop:BPBpk-Der-all-slopes} shows that it is possible that the pair $(X,Y)$ has the $\LDBPBp$ for Lipschitz compact maps without the condition that the pair $(\F(X),Y)$ has the $\BPBp$ for compact operators. Thus the condition in Theorem \ref{theorem:unifcvx} is sufficient but not necessary. Indeed, it follows from Proposition \ref{prop:BPBpk-Der-all-slopes} that  $(\R,Y)$ has the $\LDBPBp$ for Lipschitz compact maps for all range spaces $Y$, while there are Banach spaces $Y$ for which the pair $(\F(\R),Y)$ fails the $\BPBp$ for compact operators (see \cite[Remark 2.5]{AcoBerGarKimMae2}).

On the other hand, we can easily provide with versions of the above two results in terms of the derivative of $f$.
We just observe that if $f\in \Lip(\R,Y)$ is differentiable at $t_0\in \R$, then given a sufficiently small $\delta>0$, the slope $\frac{1}{\delta}\bigl[{f(t_0+\delta)-f(t_0)}\bigr]$ is close to the value of $f'(t_0)$.

\begin{cor}
Let $Y$ be a Banach space with the Radon-Nikod\'{y}m property. Then, for every $\eps>0$, if $f\in S_{\Lip(\R,Y)}$  and $t\in \R$ satisfy that $\|f'(t)\| > 1-\eps$,
then for every $\delta>0$ there exist $g\in S_{\Lip(\R,Y)}$ and $s\in \R$ such that $g$ is differentiable at $s\in \R$ with $\|g'(s)\|=1$, $\|f-g\|<\eps$ and $\|t-s\|<\eps$.
\end{cor}

\begin{cor}
Let $Y$ be a Banach space. Then, for every $\eps>0$, if $f\in S_{\Lipk(\R,Y)}$  and $t\in \R$ satisfy that $\|f'(t)\| > 1-\eps$,
then for every $\delta>0$ there exist $g\in S_{\Lipk(\R,Y)}$ and $s\in \R$ such that $g$ is differentiable at $s\in \R$ with $\|g'(s)\|=1$, $\|f-g\|<\eps$ and $\|t-s\|<\eps$.
\end{cor}

Finally, to finish the section we present the following corollaries which are straightforward consequences of Theorems \ref{theorem:unifcvx} and \ref{theorem:hilbert} for the case of general Lipschitz maps, when the range space is finite dimensional. We just recall that any finite-dimensional polyhedral Banach space $Y$ has property $\beta$, so $(\F(X),Y)$ has the BPBp for compact operators for for every Banach space $X$.

\begin{cor}\label{cor:unifcvx}
Let $X$ be a uniformly convex Banach space and let $Y$ be a finite-dimensional polyhedral Banach space. Then, $(X,Y)$ has the $\LDBPBp$ for Lipschitz maps.
\end{cor}

\begin{cor}\label{cor:hilbert}
Let $H$ be a Hilbert space and let $Y$ be a finite-dimensional polyhedral Banach space. Then, for every $\eps>0$, there exists $\eta >0$ such that whenever $f \in S_{\Lip(X,Y)}$ and $(x,y) \in \widetilde{X}$ satisfy
$$
\frac{\| f(x)-f(y) \|}{\| x-y \| } > 1 - \eta,
$$
there exist $g \in S_{\Lip(X,Y)}$, $z \in S_Y$ and $v \in X$ such that $g$ attains its norm locally directionally at the point $\bar{x}$ in the direction $\frac{x-y}{\| x-y \|}$ toward $z$, $\| g -f \| < \eps$ and $\dist (\bar{x}, [x,y] ) < \eps\max\{\|x\|,\|y\|\}$.
\end{cor}

\vspace*{0.5cm}

\noindent \textbf{Acknowledgment:\ } The authors are grateful to Rafael Chiclana, Gilles Godefroy, Gilles Lancien, and Abraham Rueda for kindly answering several inquires regarding the content of the paper.

\end{document}